\newcommand{\Hom}{\mathop{\mathrm{Hom}}\nolimits} 
\documentclass[11pt]{amsart} 
\usepackage{amssymb}
\advance \topmargin by -\headheight
\advance \topmargin by -\headsep
\evensidemargin \oddsidemargin
\newtheorem{theorem}{Theorem}[section] 
\newtheorem{lemma}[theorem]{Lemma}
\newtheorem{corollary}[theorem]{Corollary}
\newtheorem{proposition}[theorem]{Proposition}
\theoremstyle{remark} 
\newtheorem{remark}[theorem]{Remark}
\newtheorem{example}[theorem]{Example}
\title{A remark on 0-cycles as modules over algebras of finite correspondences} 
\author{M.Rovinsky} 
\address{HSE University 
(AG Laboratory, HSE, 6 Usacheva str., Moscow, Russia, 119048)
\& Institute for Information Transmission Problems of Russian Academy of Sciences}
\email{marat@mccme.ru}
\begin{document} 
\begin{abstract} Given a smooth projective variety $X$ over a field, consider the $\mathbb Q$-vector 
space $Z_0(X)$ of 0-cycles (i.e. formal finite $\mathbb Q$-linear combinations of the closed points of 
$X$) as a module over the algebra of finite correspondences. Then the rationally trivial 0-cycles 
on $X$ form an absolutely simple and essential submodule of $Z_0(X)$. \end{abstract}

\maketitle 

Let $k$ be a field. There are several ways and their versions in which the zero-cycles on $k$-schemes 
of finite type can be considered as a functor. In each of these versions, we want this functor 
to be an object of an abelian category, and we study its structure (``composition series''). 

Consider a set $S$ of smooth projective varieties over a fixed field. Let $Z_0(S)$ be 
the direct sum of the $\mathbb Q$-vector spaces of 0-cycles (i.e. formal finite linear 
combinations of the closed points) on varieties in $S$ with rational coefficients. 

We consider $Z_0(S)$ as a module over the algebra of {\sl finite correspondences}. 

The aim of this note is to show that 
the rationally trivial 0-cycles form an absolutely simple submodule 
$Z_0^{\mathrm{rat}}(S)$ of the module $Z_0(S)$, which is contained in all non-zero submodules of $Z_0(S)$. 
To some extent, this is analogous to the minimality of the rational equivalence among all `adequate' 
equivalence relations on algebraic cycles, cf. \cite[Proposition 8]{P.Samuel}. 

Assuming the Beilinson--Bloch motivic filtration conjecture, we show that the radical filtration on 
$Z_0(S)/Z_0^{\mathrm{rat}}(S)$ is an evident modification of the conjectural motivic filtration 
on Chow groups of 0-cycles. This is checked unconditionally in the case of curves. 

In the last section, a point of view on the 0-cycles on smooth, but not necessarily proper, varieties 
as a {\sl cosheaf} in an appropriate topology is briefly discussed. 

\section{Category algebras and non-degenerate modules} A category $\mathcal C$ is called 
{\sl preadditive} if, for each pair of objects $X,Y$, the morphism set $\mathrm{Hom}_{\mathcal C}(X,Y)$ 
is endowed with an abelian group structure, while the morphism composition maps 
$\circ_{X,Y,Z}:\mathrm{Hom}_{\mathcal C}(X,Y)\times\mathrm{Hom}_{\mathcal C}(Y,Z)\to
\mathrm{Hom}_{\mathcal C}(X,Z)$ are bilinear for all objects $X,Y,Z$. 

For any small preadditive category $\mathcal C$, set 
$A_{\mathcal C}:=\bigoplus_{X,Y\in\mathcal C}\Hom_{\mathcal C}(X,Y)$. 

The composition pairings $\circ_{X,Y,Z}$ (and the zero pairings between the groups $\Hom_{\mathcal C}(W,X)$ 
and $\Hom_{\mathcal C}(Y,Z)$ 
for all quadruples $W,X,Y,Z$ with $X\neq Y$) induce an associative ring structure 
on the abelian group $A_{\mathcal C}$. 

The ring $A_{\mathcal C}$ is unital if and only if there are only finitely many objects in $\mathcal C$. 
However, even if $A_{\mathcal C}$ is not unital, it is {\sl idempotented} (in the sense of 
\cite[Definition 4]{J.Bernstein}), i.e. for every finite collection $B$ of elements of $A_{\mathcal C}$ 
there is an idempotent $e\in A_{\mathcal C}$ such that $ea=a$ for all $a\in B$. Namely the sums of 
identities $\mathrm{id}_X\in\Hom_{\mathcal C}(X,X)\subseteq A_{\mathcal C}$ 
for all objects $X$ in a finite set containing the union of the supports of the elements of $B$. 
(By definition, the {\sl support} of an element $a$ is the smallest set $\mathrm{Supp}(a)$ 
such that $a\in\bigoplus_{X,Y\in\mathrm{Supp}(a)}\Hom_{\mathcal C}(X,Y)\subseteq A_{\mathcal C}$.) 

Recall, cf. e.g. \cite[p.113]{P.Cartier}, that a left module $M$ over an associative ring $A$ 
is called {\sl non-degenerate} if $AM=M$. 
Obviously, $A_{\mathcal C}$ is a non-degenerate left $A_{\mathcal C}$-module. 

Denote by $\mathrm{Mod}_{\mathcal C}$ the category of {\sl non-degenerate} left $A_{\mathcal C}$-modules. 

Denote by $\mathcal C^{\vee}$ the category of additive functors from $\mathcal C$ to the category of abelian groups. 

\begin{lemma}[Morita equivalence] If $\mathcal C$ is a small preadditive category then 
$\mathcal C^{\vee}$ and $\mathrm{Mod}_{\mathcal C}$ are equivalent abelian categories. 
In particular, if two small preadditive categories $\mathcal C$ and $\mathcal C'$ are equivalent then 
the categories $\mathrm{Mod}_{\mathcal C}$ and $\mathrm{Mod}_{\mathcal C'}$ are equivalent as well. \end{lemma} 
\begin{proof} We send any functor $\mathcal F$ from $\mathcal C$ to the category of abelian groups 
to $\bigoplus_{X\in\mathcal C}\mathcal F(X)$, 
which is a non-degenerate $A_{\mathcal C}$-module in an obvious way. 

In the opposite direction, given an $A_{\mathcal C}$-module $M$ and an object $X$, we set 
$\mathcal F(X):=\mathrm{id}_X(M)$. Any morphism $f\in\Hom_{\mathcal C}(X,X')\subseteq A_{\mathcal C}$ 
induces a map $\mathcal F(X)=\mathrm{id}_X(M)\xrightarrow{f}f\circ\mathrm{id}_X(M)=
\mathrm{id}_{X'}\circ f\circ\mathrm{id}_X(M)\subseteq\mathrm{id}_{X'}(M)=\mathcal F(X')$. 

It is easy to see that these two functors are quasi-inverse equivalences. In particular, 
we get a chain of equivalences: $\mathrm{Mod}_{\mathcal C}\simeq\mathcal C^{\vee}
\simeq(\mathcal C')^{\vee}\simeq\mathrm{Mod}_{\mathcal C'}$. \end{proof} 

The Yoneda embedding $\mathcal C\to\mathcal C^{\vee}\simeq\mathrm{Mod}_{\mathcal C}$, 
$X\mapsto h_X:=\Hom_{\mathcal C}(X,-)$ is a fully faithful functor. We are interested 
in the structure of the $A_{\mathcal C}$-module $h_X$ for the `unit' object $X$. 

\section{Algebras of finite correspondences and their modules} Fix a field $k$. 
For each pair of smooth $k$-varieties $X$ and $Y$, define $\mathrm{Cor}(X,Y)_{\mathbb Q}$ as the 
$\mathbb Q$-vector space with a basis given by the irreducible closed subsets of $X\times_kY$ whose 
associated integral subschemes are finite, flat and surjective over a connected component of $X$. 

For each triple of smooth $k$-varieties $(X,Y,Z)$, define the bilinear pairing $\mathrm{Cor}(X,Y)_{\mathbb Q}
\times\mathrm{Cor}(Y,Z)_{\mathbb Q}\xrightarrow{\circ_{X,Y,Z}}\mathrm{Cor}(X,Z)_{\mathbb Q}$ in the standard 
way: $(\alpha,\beta)\mapsto\mathrm{pr}_{XZ*}(\alpha\times Z\cap X\times\beta)$, see \cite[Ch.~1]{Fulton}. 

These pairings as compositions, turn the category of smooth $k$-varieties with morphisms 
$\mathrm{Cor}(-,-)_{\mathbb Q}$ into an {\sl additive} category, denoted $\mathrm{SmCor}_k$. 
Denote by $\mathrm{SmCor}_k^{\mathrm{proj}}$ the full subcategory of {\sl projective} $k$-varieties. 

Given a set $S$ of smooth $k$-varieties, we may consider $S$ as a full subcategory 
of $\mathrm{SmCor}_k$. As the category $S$ is preadditive, the direct sum 
$A_S:=\bigoplus_{X,Y\in S}\mathrm{Cor}(X,Y)_{\mathbb Q}$ carries a ring structure. 

\subsection{The socle of $Z_0(S)$} For each smooth variety $Y$ over $k$, let 
$Z_0(Y):=\mathrm{Cor}(\mathrm{Spec}(k),Y)_{\mathbb Q}$ be the $\mathbb Q$-vector space of 0-cycles on $Y$. 

\begin{lemma} \label{from_X_to_proj_line} Let $X$ be a smooth quasiprojective variety over $k$, 
$F$ be a characteristic zero field, and $\xi\in Z_0(Y)\otimes F$ be a non-zero 0-cycle. 
Then there exists a correspondence $\vartheta\in\mathrm{Cor}(X,\mathbb P^1_k)\otimes F$, 
such that $\vartheta\xi=[0]-[\infty]\in Z_0(\mathbb P^1_k)$. \end{lemma} 
\begin{proof} 
Let $\xi=\sum_{i=1}^Nm_i[p_i]$ for non-zero $m_i\in F$ and closed points $p_i\in X$. 

By a refinement of the projective version of the Noether normalization lemma 
proved in \cite[Theorem 1]{Kedlaya}, $X$ admits a morphism $\varphi:X\to\mathbb P^n_k$, 
where $n:=\dim X$, which maps $p_2,\dots,p_N$ into a hyperplane $H\subset\mathbb P^n_k$ and maps $p_1$ to 
the complement of $H$. Set $p_i':=\varphi(p_i)$ for all $i$, so $p_2',\dots,p_N'\in H$, 
$p_1'\in\mathbb P^n_k\smallsetminus H$. This means that $\varphi_*\xi=\sum_{i=1}^Nm_i[p_i']\neq 0$. 

Let us show by induction on $N\ge 2$ that there exists a finite endomorphism $\psi:\mathbb P^n_k\to\mathbb P^n_k$ 
sending the points $p_2',\dots,p_N'$ to a single $k$-rational point $p\in\mathbb P^n_k$ and sending 
the point $p_1'$ to a distinct $k$-rational point $q\in\mathbb P^n_k$, $q\neq p$. 
Let $W_0,\dots,W_n$ be homogeneous coordinates on $\mathbb P^n_k$ such that $H$ is given by the equation 
$W_0=0$, while both $p_2'$ and $p_3'$ do not lie on the hyperplane given by the equation $W_1=0$. 

For each $2\le i\le n$, set $w_i:=W_i/W_1$, and let $P_{ij}$ be the minimal polynomial of $w_i(p_j')$ over $k$. 

Set $d:=\max\limits_{2\le i\le n}\deg(P_{i2}P_{i3})$, and 
$P_i:=P_{i2}(w_i)P_{i3}(w_i)w_i^{d-\deg(P_{i2}P_{i3})}W_1^d$. Then the map 
\[g:(W_0:\ldots:W_N)\mapsto(W_0^d:W_1^d:P_1:P_2:\ldots:P_N)\] is a well-defined endomorphism of $\mathbb P^n_k$, 
$g$ preserves $H$, the point $g(p_2')=g(p_3')$ is $k$-rational, and $g$ transforms $\varphi_*\xi$ 
to $m_1[p_1'']+(m_2+m_3)[p_3'']+\sum_{i=4}^Nm_i[p_i'']$, where $p_3'',\dots,p_N''\in H$ and $p_1''\notin H$. 

Then $\psi_*\varphi_*\xi$ is a non-zero multiple of $[p]-[q]$. 

Let $\Upsilon$ be an $n$-dimensional variety admitting a non-constant morphism $h:\Upsilon\to\mathbb P^1_k$ 
(e.g., $\Upsilon=\mathbb P^{n-1}\times\mathbb P^1_k$ and $h:\Upsilon\to\mathbb P^1_k$ is the projection). 
Fix a fibre $D$ of $h$, and a hyperplane $H'\subset\mathbb P^n_k$ containing $p$ but not $q$. By the same 
\cite[Theorem 1]{Kedlaya}, there exists a finite morphism $\pi:\Upsilon\to\mathbb P^n_k$ such that $\pi(D)=H'$, 
so $D$ meets $\pi^{-1}(p)$ but not $\pi^{-1}(q)$, and therefore, $h_*\pi^*\psi_*\varphi_*\xi\neq 0$. 
Then $h_*({}^t\Gamma_{\pi})_*\psi_*\varphi_*\xi=h_*\pi^*\psi_*\varphi_*\xi$ is a non-zero divisor 
$E=\sum\limits_{i=0}^na_i[q_i]$ on $\mathbb P^1_k$ for some $a_i\neq 0$ and pairwise distinct $q_i$. 

Choose a morphism $f:\mathbb P^1_k\to\mathbb P^1_k$ such that $f(q_0)=0$, $f(q_i)=\infty$ for all 
$1\le i\le n$, so $f_*h_*({}^t\Gamma_{\pi})_*\psi_*\varphi_*\xi=a_0([0]-[\infty])$. \end{proof}

For each set $S$ of smooth varieties over $k$, consider $Z_0(S):=\bigoplus_{X\in S}Z_0(X)$. Then 
the above pairings $\circ_{\mathrm{Spec}(k),Y,Z}:Z_0(Y)\times\mathrm{Cor}(Y,Z)_{\mathbb Q}\to Z_0(Z)$, 
given by $(\alpha,\beta)\mapsto\mathrm{pr}_{Z*}(\alpha\times Z\cap\beta)$, induce an 
$A_S$-module structure on $Z_0(S)$. 

Define the {\sl degree} of a 0-cycle $\alpha=\sum_im_iP_i$ on $X$ by 
$\deg(\alpha):=\sum_im_i[\varkappa(P_i):k]$, where $\varkappa(P_i)$ is the residue field of $P_i$. 

For each smooth variety $Y$ over $k$, let $Z_0^{\circ}(Y)$ be the subspace of 0-cycles of degree 0 
on each connected component of $Y$. 

Obviously, $Z_0^{\circ}(S):=\bigoplus\limits_{X\in S}Z_0^{\circ}(X)$ is an $A_S$-submodule of $Z_0(S)$. 

Recall (\cite[\S2]{P.Samuel}, \cite[Ch.~1]{Fulton}), that a cycle is called {\sl rationally equivalent} 
to zero (or {\sl rationally trivial}) if it is a sum of divisors of rational functions on subvarieties. 

\begin{theorem} \label{the-socle} Let $S$ be a set of smooth varieties over $k$, 
and $F$ be a characteristic zero field. 
Then \begin{enumerate} \item \label{radical} any proper $(A_S\otimes F)$-submodule of 
$Z_0(S)\otimes F$ is contained in the submodule $Z_0^{\circ}(S)\otimes F$; 
\item \label{socle} if $S$ consists of projective varieties then any non-zero $(A_S\otimes F)$-submodule 
of $Z_0(S)\otimes F$ contains the $A_S$-submodule 
\[Z_0^{\mathrm{rat}}(S):=\bigoplus\limits_{X\in S}Z_0^{\mathrm{rat}}(X)\] 
of 0-cycles rationally equivalent to 0 on all $X\in S$. 
\end{enumerate} \end{theorem} 
\begin{proof} It is clear that if $S'$ is the set of connected components of 
varieties in $S$ then $A_{S'}$ and $A_S$ are naturally isomorphic, while $Z_0(S')$ and $Z_0(S)$ coincide 
as $A_S$-modules. This means that we may assume that all varieties in $S$ are connected. Given any 
characteristic zero field $F$ and any non-zero element $\xi=(\xi_X)_{X\in S}\in Z_0(S)\otimes F$, 
there is $X\in S$ such that $\xi_X\neq 0$, so $\xi':=\mathrm{id}_X\xi\neq 0$. 
\begin{enumerate} \item For any $Y\in S$ and any closed point $y\in Y$, the finite correspondence 
$[X\times_ky]\in\mathrm{Cor}(X,Y)_{\mathbb Q}$ maps $\xi_X$ to the 0-cycle 
$\deg(\xi_X)\cdot[y]\in Z_0(Y)\otimes F$, so if $\deg(\xi_X)\neq 0$ then $\xi'$ 
(and therefore, $\xi$) generates the whole $(A_S\otimes F)$-module $Z_0(S)\otimes F$, 
which is equivalent to (\ref{radical}). 

\item According to (\ref{radical}), we may further assume that $\deg(\xi_X)=0$ and, as $\xi_X\neq 0$, 
that $\dim X>0$. 

By Lemma~\ref{from_X_to_proj_line}, there exists a correspondence 
$\vartheta\in\mathrm{Cor}(X,\mathbb P^1_k)\otimes F$, such that $\vartheta\xi_X=[0]-[\infty]\in Z_0(\mathbb P^1_k)$.

Finally, for each $Y\in S$, any 0-cycle on $Y$ rationally equivalent to 0 is a linear combination of images 
of the cycle $[0]-[\infty]$ under finite correspondences $\gamma$ from $\mathbb P^1_k$ to $Y$, i.e. of elements 
$(\gamma\circ\vartheta)_*\xi_X$ for appropriate $\gamma$'s. \end{enumerate} \end{proof} 

\begin{remark} A module $M$ over a $\mathbb Q$-algebra $A$ is called {\sl absolutely simple} if 
$M\otimes F$ is a simple $(A_S\otimes F)$-module for any characteristic zero field $F$. 
Equivalently, the $A$-module $M$ is simple and $\mathrm{End}_A(M)=\mathbb Q$. In particular, 
in the setting of Theorem~\ref{the-socle}, the $A_S$-modules $Z_0(S)/Z_0^{\circ}(S)$ and 
$Z_0^{\mathrm{rat}}(S)$ are absolutely simple, whenever $S$ is non-empty. \end{remark} 

\subsection{Motivic $A_S$-modules} By definition (\cite{P.Samuel}), an equivalence relation 
$\sim$ is {\sl adequate} if it satisfies the following conditions: 
\begin{itemize} \item it is compatible with the addition of cycles, i.e. a subgroup $Z^{\sim}(X)$ of 
cycles on each variety $X$ is fixed, and two cycles on $X$ are equivalent if and only if their 
difference belongs to $Z^{\sim}(X)$; 
\item for any variety $X$, any cycle $\alpha$ on $X$, and any subvariety $W$ on $X$, there exists 
a cycle $\alpha'\sim\alpha$ intersecting $W$ properly; 
\item for any pair of smooth projective varieties $X$ and $Y$, a cycle $\beta\sim 0$ on $X$, and 
a cycle $\alpha$ on $X\times Y$ intersecting $\beta\times Y$ properly, the cycle 
$\alpha(\beta):=\mathrm{pr}_Y((\beta\times Y)\cdot\alpha)$ is $\sim$-equivalent to 0 on $Y$. \end{itemize} 
\begin{example}[\cite{P.Samuel}, \S2] Besides the {\sl rational equivalence} mentioned above, the following  
equivalence relations are adequate. \begin{itemize} 
\item A cycle $\alpha$ on a smooth projective variety $X$ is called {\sl algebraically equivalent} 
to zero if there exist a curve $C$, points $c,d\in C$ and a cycle $\beta$ on $X\times C$, which is 
flat over $C$ such that $\alpha=[\beta\cap(X\times\{c\})]-[\beta\cap(X\times\{d\})]$. 
\item A cycle on a smooth projective variety is called {\sl homologically equivalent} to zero 
(with respect to a fixed Weil cohomology theory) if it is annihilated by the cycle map. 
\item A cycle $\alpha$ on a smooth projective variety is called {\sl numerically equivalent} to zero if 
$\deg(\alpha\cap W)=0$ for any subvariety $W$ of the complementary dimension that meet $\alpha$ properly. 
\end{itemize} \end{example} 

Recall (see, e.g., \cite{Yu.I.Manin}), that a (homological) effective 
{\sl Grothendieck motive} over $k$ modulo an `adequate' equivalence relation $\sim$ is defined as 
a pair $(X,\pi)$ consisting of a smooth projective variety $X$ over $k$ and a projector $\pi$ in 
the algebra of self-correspondences on $X$ of dimension $\dim X$ with coefficients in $\mathbb Q$ 
modulo $\sim$. The morphisms between pairs $(X,\pi)$ and $(X',\pi')$ are algebraic cycles $\alpha$ 
on $X\times_kX'$ of dimension $\dim X$ modulo $\sim$, and such that $\alpha=\pi'\circ\alpha\circ\pi$. 

The motives over $k$ modulo an equivalence relation $\sim$ form a pseudo-abelian category, denoted 
by $\mathcal{M}_{k,\mathrm{eff}}^{\sim}$. The category $\mathcal{M}_{k,\mathrm{eff}}^{\sim}$ 
carries a tensor structure: $(X,\pi)\otimes(X',\pi'):=(X\times_kX',\pi\times\pi')$. 

Denote by $\mathbb M^{\sim}:\mathrm{SmCor}_k^{\mathrm{proj}}\to\mathcal{M}_{k,\mathrm{eff}}^{\sim}$ the 
additive functor $X\mapsto(X,\Delta_X)$, where $\Delta_X$ is the class of the diagonal in $X\times_kX$. 
In particular, \[\mathbb M^{\sim}(\mathbb P^1_k)\cong\mathbb M^{\sim}(\mathrm{Spec}(k))\oplus\mathbb L,\quad
\mbox{where $\mathbb L=(\mathbb P^1_k,[\{q\}\times\mathbb P^1_k])$}\] 
for any rational point $q\in\mathbb P^1(k)$. It is easy to see that the natural map 
\[\mathrm{Hom}_{\mathcal{M}_{k,\mathrm{eff}}^{\sim}}(U,V)\to
\mathrm{Hom}_{\mathcal{M}_{k,\mathrm{eff}}^{\sim}}(U\otimes\mathbb L,V\otimes\mathbb L)\] 
is bijective for all effective motives $U$ and $V$. 

Denote by $\mathcal{M}_k^{\sim}$ the category of triples $(X,\pi,n)$, 
where $(X,\pi)$ are as above and $n$ is an integer, while 
$\mathrm{Hom}_{\mathcal{M}_k^{\sim}}((X,\pi,n),(X',\pi',n')):=\mathrm{Hom}_{\mathcal{M}_{k,\mathrm{eff}}^{\sim}}
((X,\pi)\otimes\mathbb L^{\otimes(m+n-n')},(X',\pi')\otimes\mathbb L^{\otimes m})$ 
for any integer $m>|n'-n|$. We consider $\mathcal{M}_{k,\mathrm{eff}}^{\sim}$ 
as a full subcategory of $\mathcal{M}_k^{\sim}$ under the embedding $(X,\pi)\mapsto(X,\pi,0)$. 

For each variety $Y$ and an integer $q$, denote by $CH_q(Y)$ the group of 
dimension $q$ cycles on $Y$ modulo the rational equivalence. 
\begin{theorem} The functor $\mathbb M^{\sim}$ is full. In other words, the natural ring homomorphism 
$A_S\to\bigoplus_{X,Y\in S}CH_{\dim X}(X\times_kY)_{\mathbb Q}$ is surjective for any set $S$ of 
smooth projective varieties over $k$. \end{theorem} 
\begin{proof} This is a particular case of \cite[Theorem 7.1]{FriedlanderVoevodsky}. \end{proof} 

For any set $S$ of smooth projective vaieties over $k$, each Grothendieck motive 
$N\in\mathcal{M}_k^{\sim}$ gives rise to an $A_S$-module 
$\mathfrak{M}^{\sim}_N(S):=\bigoplus_{X\in S}\Hom_{\mathcal{M}_k^{\sim}}(N,\mathbb M^{\sim}(X))$. 
 
We omit the symbol $\sim$ from the notation when $\sim=\sim_{\mathrm{num}}$ is the {\sl numerical} equivalence. 

\begin{corollary} For any motive $N\in\mathcal{M}_k$, the $A_S$-module $\mathfrak{M}_N(S)$ is semisimple. 
\end{corollary} 
\begin{proof} The $A_S$-action on $\mathfrak{M}_N(S)$ factors through an action of the algebra 
$A_S/\sim_{\mathrm{num}}$, while 
$A_S/\sim_{\mathrm{rat}}\cong\bigoplus_{X,Y\in S}CH_{\dim X}(X\times_kY)_{\mathbb Q}$, so 
$A_S/\sim_{\mathrm{num}}\cong\bigoplus_{X,Y\in S}CH_{\dim X}(X\times_kY)_{\mathbb Q}/\sim_{\mathrm{num}}$. 

By \cite{Jannsen-semisimplicity}, $\mathcal{M}_k$ is an abelian semisimple category, and therefore, any 
non-degenerate\\ $(A_S/\sim_{\mathrm{num}})$-module is semisimple. In particular, so is the 
$A_S$-module $\mathfrak{M}_N(S)$. \end{proof} 

\section{Loewy filtrations on $Z_0(S)$} Modifying slightly the standard definition (see, e.g. \cite{Irving}), 
a filtration of a module $M$ is called a {\sl Loewy filtration} if it is finite, its successive 
quotients are semisimple and its length is minimal under these assumptions. 

Let $S$ be a set of smooth irreducible projective varieties over a field $k$. 
We are interested in Loewy filtrations on the $A_S$-module $Z_0(S)$. 

By Theorem~\ref{the-socle}, the socle (i.e. the maximal semisimple submodule) of the $A_S$-module $Z_0(S)$ 
is $Z_0^{\mathrm{rat}}(S)$, while the radical (i.e. the intersection of all maximal submodules) of the 
$A_S$-module $Z_0(S)$ is $Z_0^{\circ}(S)$, and $Z_0^{\mathrm{rat}}(S)$ is an essential submodule of $Z_0(S)$. 

The $A_S$-action on the quotient $CH_0(S)_{\mathbb Q}:=Z_0(S)/Z_0^{\mathrm{rat}}(S)$ factors through 
an action of the quotient $A_S/\sim_{\mathrm{rat}}$ of $A_S$ by the rational equivalence. 

\subsection{The case of curves} 
\begin{proposition} Let $S$ be a set of smooth projective curves over $k$. 

Then $Z_0^{\mathrm{rat}}(S)\subset Z_0^{\circ}(S)\subset Z_0(S)$ is 
the unique Loewy filtration on the $A_S$-module $Z_0(S)$. \end{proposition} 
\begin{proof} By Theorem~\ref{the-socle}, the socle of the $A_S$-module $Z_0(S)$ is simple and coincides 
with $Z_0^{\mathrm{rat}}(S)$, while $Z_0^{\circ}(S)$ is the unique maximal submodule of the $A_S$-module 
$Z_0(S)$. There remains only to check the semisimplicity of $Z_0^{\circ}(S)/Z_0^{\mathrm{rat}}(S)$. 

One has $A_S/\sim_{\mathrm{rat}}=\bigoplus_{X,Y\in S}\mathrm{Pic}(X\times_kY)_{\mathbb Q}$. Then the subgroup 
\[I:=\bigoplus_{X,Y\in S}\mathrm{Pic}^{\circ}(X\times_kY)_{\mathbb Q}\] is an ideal in $A_S/\sim_{\mathrm{rat}}$ 
with $I^2=0$, while $(A_S/\sim_{\mathrm{rat}})/I=\bigoplus_{X,Y\in S}\mathrm{NS}(X\times_kY)_{\mathbb Q}$ is a 
semisimple algebra. Here $\mathrm{Pic}$ is the Picard group, $\mathrm{Pic}^{\circ}$ is the subgroup 
of algebraically trivial elements, $\mathrm{NS}:=\mathrm{Pic}/\mathrm{Pic}^{\circ}$ is the N\'eron--Severi group. 

Then, for any $(A_S/\sim_{\mathrm{rat}})$-module $M$, the submodule $IM$ and the quotient $M/IM$ 
can be considered as $(A_S/\sim_{\mathrm{rat}})/I$-modules, and thus, they are semisimple. 
Applying this to the module $M=Z_0(S)/Z_0^{\mathrm{rat}}(S)$, we see that the $A_S$-module 
$IM=Z_0^{\circ}(S)/Z_0^{\mathrm{rat}}(S)$ is semisimple. \end{proof} 

\subsection{Consequences of the filtration conjecture} According to the Bloch--Beilinson motivic filtration 
conjecture (e.g., \cite[Conjecture 2.3]{Jannsen}, \cite[Conjecture 33]{M.Levine}), there should exist 
a neutral tannakian $\mathbb Q$-linear category $\mathcal{MM}_k$ (of {\sl mixed motives} over $k$) 
containing the category $\mathcal{M}_k$ as the full subcategory of the semisimple objects, covariant 
functors $H_i(-,\mathbb Q(j))$ (homology; $i,j\in\mathbb Z$) from the category of varieties over $k$ 
to $\mathcal{MM}_k$, and a functorial descending filtration $\mathcal F^{\bullet}$ on the Chow groups 
$CH_q(X)_{\mathbb Q}$ for smooth projective $k$-varieties $X$ such that 
$\mathcal F^0CH_q(X)_{\mathbb Q}=CH_q(X)_{\mathbb Q}$ and 
\[gr^i_{\mathcal F}CH^q(X)_{\mathbb Q}=\mathrm{Ext}^i_{\mathcal{MM}_k}(\mathbb Q(0),H_{2q+i}(X,\mathbb Q(-q))).\] 

As a part of the filtration conjecture, it is natural to assume the Grothendieck's `semisimplicity 
conjecture' on the coincidence of homological $\otimes\mathbb Q$ and numerical equivalences, 
so that the motive $H_{2q+i}(X,\mathbb Q(-q))$ is semisimple 
by U.~Jannsen's  theorem, \cite{Jannsen-semisimplicity}. 

A simple effective motive $P\in\mathcal{M}_k$ is called {\sl primitive of weight} 
$-i\le 0$ if, (i) $P\cong(X,\pi)$ for some $X,\pi$ with $\dim X=i$, and 
(ii) $\mathrm{Hom}_{\mathcal{M}_k}(P,\mathbb M(Y\times\mathbb P^1))=0$ 
for any smooth projective variety $Y$ of dimension $<i$. 

In particular, when $q=0$ the Beilinson 
formula becomes \begin{multline*}gr^i_{\mathcal F}CH_0(X)_{\mathbb Q}
=\mathrm{Ext}^i_{\mathcal{MM}_k}(\mathbb Q(0),H_i(X,\mathbb Q))\\ 
=\bigoplus_P\mathrm{Ext}^i_{\mathcal{MM}_k}(\mathbb Q(0),P)\otimes_{\mathrm{End}_{\mathcal{M}_k}(P)}
\mathrm{Hom}_{\mathcal{MM}_k}(P,H_i(X,\mathbb Q)),\end{multline*} 
where $P$ runs over the isomorphism classes of simple primitive motives of weight $-i$, and 
we see that the spaces $\mathcal F^iCH_0(X)_{\mathbb Q}$ should be covariant functorial. 

For each set $S$ of smooth irreducible projective varieties over a field $k$, and each integer $i\ge 0$, 
consider $\mathcal F^iCH_0(S)_{\mathbb Q}:=\bigoplus_{X\in S}\mathcal F^iCH_0(X)_{\mathbb Q}$. 
By the functoriality of $\mathcal F^{\bullet}$, this is an $A_S$-submodule of $CH_0(S)_{\mathbb Q}$. 

The algebra $A_S$ acts on $gr^i_{\mathcal F}CH_0(S)_{\mathbb Q}$ via its action on the motives 
$H_i(X,\mathbb Q)$, so the $A_S$-action on $gr^i_{\mathcal F}CH_0(S)_{\mathbb Q}$ factors 
through an action of the quotient $A_S/\sim_{\mathrm{num}}$ of $A_S$, i.e. of the algebra 
$B_S:=\bigoplus_{X,Y\in S}CH_{\dim X}(X\times_kY)_{\mathbb Q}/\sim_{\mathrm{num}}$. As the algebra 
$B_S$ is semisimple, the $A_S$-module $gr^i_{\mathcal F}CH_0(S)_{\mathbb Q}$ is semisimple as well. 

In particular, if dimensions of the varieties in $S$ do not exceed $d$ then the length $\ell(S)$ of any 
Loewy filtration of $CH_0(S)_{\mathbb Q}$ does not exceed $d+1$. (More precisely, $\ell(S)-1$ does not exceed 
the number of those $0\le i\le d$ for which $H_i(X,\mathbb Q)$ is not a Tate twist of an effective motive 
of weight $>-i$ for at least one $X\in S$.) 

It seems that the {\sl radical filtration} on $CH_0(S)_{\mathbb Q}$ (i.e. the strictly descending 
sequence of the iterated radicals) is the motivic one, but with the repeating terms omitted.  

\begin{remark} Usually (e.g., \cite{corresp_at_gen_pt} or 
\cite[Conjecture 2.3]{Jannsen}, \cite[\S5.3]{M.Levine}) one states the motivic conjectures in 
the contravariant setting, i.e. instead of $\mathcal{M}_k$ one considers its dual category (which 
is in fact equivalent to $\mathcal{M}_k$), while the homology functors from the category 
of varieties over $k$ to $\mathcal{MM}_k$ are replaced by contravariant functors $H^i(-,\mathbb Q(j))$. 
Then the homological object \[H_i(X,\mathbb Q):=H^{2\dim X-i}(X,\mathbb Q(\dim X))\]
is the Poincar\'e dual of the cohomological object $H^i(X,\mathbb Q)$, 
while the Beilinson formula for {\sl codimension} $q$ Chow groups of smooth projective $k$-varieties 
$X$ can be rewritten as \[gr^i_{\mathcal F}CH^q(X)_{\mathbb Q}
=\mathrm{Ext}^i_{\mathcal{MM}_k}(\mathbb Q(0),H^{2q-i}(X,\mathbb Q(q))).\] \end{remark}

\section{Correspondences on non-proper varieties?} 
One could try to extend Theorem~\ref{the-socle}.\ref{socle} to collections $S$ of smooth varieties over $k$ 
that are not necessarily proper. However, as there are no non-constant morphisms from projective varieties to 
affine ones, it seems that the structure of the $A_S$-module $Z_0(S)$ may be quiet complicated. 

On the other hand, if the set $S$ is considered as a preadditive category then the $A_S$-modules 
become {\sl precosheaves with transfers} (in analogy with the terminology of V.~Voevodsky). 
To restrict the category of precosheaves one can pass to the category of cosheaves in 
such a non-trivial Grothendieck topology where $Y\mapsto Z_0(Y)$ is a cosheaf. 

In \cite{Voevodsky_Homology-I}, a Grothendieck topology on the categories of schemes of finite type 
over noetherian bases, called the $h$-{\sl topology}, is defined, see also \cite[\S10]{SuslinVoevodsky}. 
This topology is generated by a pretopology, where the coverings are those finite families $(p_i:U_i\to X)_i$ of 
morphisms of finite type that $\amalg p_i:\amalg U_i\to X$ 
are universal topological epimorphisms (i.e. a subset of $X$ is open if and only if so is its preimage, 
and any base change has the same property). 

A precosheaf $\mathcal F$ of abelian groups on the category of schemes 
of finite type over $k$ is an $h$-{\sl cosheaf} if the sequence 
$\mathcal F(U\times_XU)\xrightarrow{f_*\circ\mathrm{pr}_{1*}-f_*\circ\mathrm{pr}_{2*}}
\mathcal F(U)\xrightarrow{f_*}\mathcal F(X)\to 0$ is exact for any $h$-covering $f:U\to X$. 
By an $h$-cosheaf on the category of smooth varieties over $k$ we mean the restriction 
of an $h$-cosheaf on the category of schemes of finite type over $k$. 

The following lemma is related somehow to \cite[Prop.3.1.3]{Voevodsky}, where $f$ is a Nisnevich cover. 
\begin{lemma} \label{Z_0_is_h-cosheaf} If a quasi-compact morphism of schemes $Y\xrightarrow{f}X$ 
is surjective $($on the sets of points$)$ then \begin{itemize} \item it is surjective on the sets of closed points; 
\item the sequence $Z_0(Y\times_XY)\xrightarrow{f_*\circ\mathrm{pr}_{1*}-f_*\circ\mathrm{pr}_{2*}}Z_0(Y)
\xrightarrow{f_*}Z_0(X)\to 0$ is exact. In particular, $X\mapsto Z_0(X)$ is an $h$-cosheaf. 
\end{itemize} \end{lemma} 
\begin{proof} Let $p$ be a closed point of $X$. Then $Y_p:=f^{-1}(p)$ is a non-empty closed subset of $Y$, 
so it suffices to show the existence of a closed point of $Y_p$. Suppose on the contrary that there are 
no closed points in $Y_p$. As $Y_p$ is quasi-compact, it can be covered by a finite collection $S$ of 
affine opens: $Y_p=\bigcup_{U\in S}U$. Let us construct recursively a sequence of points $q_i\in X$ and 
a sequence $U_1,U_2,\dots$ of elements of $S$ as follows: let $U_1$ be an arbitrary element of $S$, $q_1$ 
be an arbitrary closed point of $U_1$; for $i>1$, if the closure $\overline{\{q_{i-1}\}}$ of $q_{i-1}$ 
is not contained in $U_1\cup\dots\cup U_{i-1}$, let (i) $q_i'$ be a point of $\overline{\{q_{i-1}\}}$ 
in the complement of $\bigcup_{j=1}^{i-1}U_j$, (ii) $U_i$ be an element of $S$ containing $q_i'$, 
(iii) $q_i$ be a closed point of $\overline{\{q_i'\}}\cap U_i$: $\overline{\{q_i\}}\cap U_i=\{q_i\}$. 

Then $\overline{\{q_i\}}\cap(U_1\cup\dots\cup U_j)$ is a subset of 
$\{q_1,\dots,q_j\}$ for any $j\le i$. 
As $S$ is finite, there is some $1\le n\le\#S$ such that $\overline{\{q_n\}}$ is contained 
in $U_1\cup\dots\cup U_n$. As the complement of $\bigcup_{j=1}^{n-1}U_j$ is closed, the set 
$\{q_n\}=\overline{\{q_n\}}\cap(X\smallsetminus(\bigcup_{j=1}^{n-1}U_j))\subseteq U_n$ 
is closed as well. 

The kernel of $f_*$ is spanned by the elements $q-q'$ for all closed points $q,q'$ of $Y$ such that 
$f(q)=f(q')$. But $q-q'$ is the image of any closed point of $q\times_Xq'\subseteq Y\times_XY$. \end{proof} 

\begin{remark} The proof of Lemma~\ref{Z_0_is_h-cosheaf} can be obviously modified to show that 
the linear combinations of $k$-rational points on $k$-schemes of finite type ($X\mapsto\mathbb Q[X(k)]$) 
form a Nisnevich subcosheaf 
without transfers (i.e. functorial with respect to the morphisms of schemes, not with respect to the finite 
correspondences) of the $h$-cosheaf with transfers $Z_0:X\mapsto Z_0(X)$. \end{remark} 

Lemma~\ref{Z_0_is_h-cosheaf} suggests that, in the non-proper case, the category of $h$-cosheaves is 
more appropriate than the much bigger category of $A_S$-modules. Then the natural guess is that the 
socle $\mathrm{Soc}(Z_0)$ of the $h$-cosheaf $Z_0$ is simple and consists of those 0-cycles that become 
rationally trivial on some smooth compactifications, while the radical filtration on $Z_0/\mathrm{Soc}(Z_0)$ 
is separable and coincides with the motivic one.

\vspace{5mm}

\noindent 
{\sl Acknowledgements.} {\small The study has been funded within the framework of the HSE University
Basic Research Program. Discussions with Vadim Vologodsky, Sergey Gorchinskiy, Dmitry Kaledin, 
and especially with Ivan Panin, were very helpful for me.}


\vspace{5mm}

\begin{thebibliography}{}
\bibitem{corresp_at_gen_pt} A.Beilinson, {\em Remarks on $n$-motives and correspondences at the generic point,} 
in Motives, polylogarithms and Hodge theory, Part I (Irvine, CA, 1998), volume 3 of Int. Press Lect. Ser. 
Int. Press, Somerville, MA, 2002, 35--46. 
\bibitem{J.Bernstein} J.Bernstein, Notes of lectures on Representations of $p$-adic groups, 
Harvard University, Fall 1992, written by Karl E. Rumelhart, {\tt http://www.math.tau.ac.il/\~{}bernstei/} 
\bibitem{P.Cartier} P.Cartier, {\em Representations of $\mathfrak{p}$-adic groups: A survey,} 111--155. 
In Proc. of Symp. in Pure Math. \textbf{33}, Part 1 Automorphic Forms, Representations, 
and $L$-functions, A.Borel, W.Casselman (Eds.) 1977. 
\bibitem{FriedlanderVoevodsky} E.M.Friedlander, V.Voevodsky, {\em Bivariant cycle cohomology,} 
in V.Voevodsky, A.Suslin, E.M.Friedlander, Cycles, transfers, and motivic homology theories. 
Ann. Math. Studies {\bf 143} (2000), Princeton Univ. Press, 138--187. 
\bibitem{Fulton} W.Fulton, Intersection Theory, Springer, 1984. 
\bibitem{Irving} R.S.Irving, {\em The socle filtration of a Verma module,} Ann. Sci. ENS, 4e s\'erie, 
\textbf{21} (1988), 47--65. 
\bibitem{Jannsen-semisimplicity} U.Jannsen, {\em Motives, numerical equivalence, and semi-simplicity,} 
Invent. Math. \textbf{107} (1992), 447--452. 
\bibitem{Jannsen} U.Jannsen, {\em Motivic sheaves and filtrations on Chow groups,} in Motives,
U.Jannsen, S.Kleiman,\\ J.-P.Serre (Eds.), Proc. of Symp. in Pure Math. \textbf{55} (1994), Part 1, 245--302. 
\bibitem{Kedlaya} K.S.Kedlaya, {\it More \'etale covers of affine spaces in positive characteristic,} 
J. Algebraic Geom. {\bf 14} (2005), 187--192. 
\bibitem{M.Levine} M.~Levine, {\em Mixed Motives,} in Handbook of K-Theory, Vol. 1, 
E.M.~Friedlander, D.R.~Grayson (Ed.), Springer-Verlag, Berlin Heidelberg, 2005, 429--521. 
\bibitem{Yu.I.Manin} Yu.I.Manin, 
{\em Correspondences, motifs and monoidal transformations,} Math. USSR-Sb., \textbf{6:4} (1968), 439--470. 
\bibitem{P.Samuel} P.Samuel, {\em Relations d'\'equivalence en g\'eom\'etrie alg\'ebrique,} 
Proc. Internat. Congress Math. Edinburgh 1958, Cambridge University Press, Cambridge (1960), 470--487.
\bibitem{SuslinVoevodsky} A.Suslin, V.Voevodsky, {\em Singular homology of abstract algebraic varieties,} 
Invent. math. \textbf{123} (1996), 61--94. 
\bibitem{Voevodsky_Homology-I} V.Voevodsky, {\em Homology of schemes}, Selecta Mathematica, New Series
\textbf{2}, No. 1 (1996), 111--153. 
\bibitem{Voevodsky} V.Voevodsky, {\em Triangulated categories of motives over a field,} 
in V.Voevodsky, A.Suslin, E.M.Friedlander, Cycles, transfers, and motivic homology theories. 
Ann. Math. Studies \textbf{143} (2000), Princeton Univ. Press, 188--238. 
\end{thebibliography}
\end{document}